\documentclass[10pt]{article}

\usepackage[T1]{fontenc}%
\usepackage{amsmath}
\usepackage{amsfonts}
\usepackage{amssymb}

\newtheorem{theorem}{Theorem}

\newtheorem{lemma}[theorem]{Lemma}

\newtheorem{proposition}[theorem]{Proposition}

\newenvironment{proof}[1][Proof]{\noindent \textbf{#1.} }{\ \rule{0.5em}{0.5em}}

\newcommand{\R}{\mathbb{R}}

\newcommand{\dee}{\mathop{\! \,\mathrm{d} \!}\nolimits}
\newcommand{\comp}{\raisebox{0pt}{$\scriptstyle\circ \, $}}
\newcommand{\setrule}{\, \rule[-4pt]{.5pt}{13pt}\, }

\allowdisplaybreaks

\begin{document}
\title{\bf \large On Subcartesian Spaces\\
Leibniz's Rule Implies the Chain Rule}
\author{Richard Cushman and J\k{e}drzej \'{S}niatycki\thanks{Department 
of Mathematics and Statistics, University of Calgary. \newline
\rule{15pt}{0pt}email: rcushman@ucalgary.ca and sniatycki@ucalgary.ca}}
\date{}
\maketitle
\addtocounter{footnote}{1}
\footnotetext{printed: 24 March 2019}

\begin{abstract}
We show that derivations of the differential structure of a subcartesian space
satisfy chain rule and have maximal integral curves.  
\end{abstract}\medskip

\section{Introduction}

The structure of a smooth manifold is usually described in terms of its
complete atlas. In 1967, Aronszajn \cite{aronszajn} applied this description
to Hausdorff spaces that are locally diffeomorphic to arbitrary subsets of $%
\mathbb{R}^{n}$, which he called subcartesian spaces. In 1973, Walczak \cite%
{walczak 1973} showed that that subcartesian spaces of Aronszajn are special
cases of differential spaces introduced by Sikorski \cite{sikorski 1967}.
This implied that the geometric structure of a subcartesian space $S$ can be
completely described by its ring of smooth functions $\mathcal{C}^{\infty
}(S)$.\smallskip 

In recent years, the notion of $\mathcal{C}^{\infty }$-rings and $\mathcal{C}%
^{\infty }$-ringed spaces appeared as part of Spivak's definition of derived
manifolds \cite{spivak 2010}. Joyce \cite{joyce 2012} developed an alternative theory of
derived differential geometry going beyond Spivak's derived manifolds.\smallskip 

The definition of derivations of $\mathcal{C}^{\infty }$-rings are required
to satisfy chain rule, while derivations of the differential structure $%
\mathcal{C}^{\infty }(S)$ of a differential space $S$ are defined
algebraically in terms of Leibniz's rule. We show that, if $S$ is
subcartesian, the derivations of $\mathcal{C}^{\infty }(S)$ also satisfy the
chain rule. This ensures that subcartesian spaces do not require the additional
assumption that their differential structures are $\mathcal{C}^{\infty }$%
-rings. In particular, this justifies integration of derivations of
differential structures of subcartesian spaces studied in \cite{sniatycki
2013}.

\section{Differential spaces}

A \emph{differential structure} on a topological space $S$ is a family 
$\mathcal{C}^{\infty }(S)$ of real valued functions on $S$ satisfying the following
conditions.\medskip

\par 1. \parbox[t]{4in}{The family 
$\{f^{-1}(I)\setrule \,  f\in \mathcal{C}^{\infty }(S),\,\mbox{and $I$ is an
open interval in $\mathbb{R}$} \}$ is a sub-basis for the topology of $S$.}\smallskip 

\par  2. \parbox[t]{4in}{If $f_{1},...,f_{n}\in \mathcal{C}^{\infty }(S)$ and $F\in 
\mathcal{C}^{\infty }(\mathbb{R}^{n})$, then $F(f_{1},...,f_{n})\in \mathcal{C}^{\infty }(S)$.} 
\smallskip 

\par  3. \parbox[t]{4in}{If $f:S\rightarrow \mathbb{R}$ is a function such that, for
every $x\in S$, there exists an open neighbourhood $U$ of $x$ and a function 
$f_{x}\in \mathcal{C}^{\infty }(S)$ satisfying 
\[
f_x|_U=f|_U
\]%
then $f\in \mathcal{C}^{\infty }(S).$ Here the vertical bar $\mid $
denotes restriction.}\medskip  

Functions $f\in \mathcal{C}^{\infty }(S)$ are called \emph{smooth} functions on $S$. 
It follows from condition 1 that smooth functions on $S$ are continuous. 
Condition 2 with $F(f_{1},f_{2})=af_{1}+bf_{2}$, where $a,b\in \mathbb{R}$
implies that $\mathcal{C}^{\infty }(S)$ is a vector space. Similarly, taking 
$F(f_{1},f_{2})=f_{1}f_{2}$, we conclude that $\mathcal{C}^{\infty }(S)$ is
closed under multiplication of functions. A topological space $S$ endowed
with a differential structure is called a \emph{differential space}.\smallskip

In his original definition, Sikorski \cite{sikorski 1972} defined 
$\mathcal{C}^{\infty }(S)$ to be a family of functions satisfying condition 2. Then, he
used condition 1 to define topology on $S$. Finally, he imposed condition 3
as a consistency condition. \medskip

A map $\varphi :R\rightarrow S$ is \emph{smooth}  if $\varphi ^{\ast }f=f\comp
\varphi \in \mathcal{C}^{\infty }(R)$ for every $f\in \mathcal{C}^{\infty}(S)$. 
A smooth map $\varphi $ between differential spaces is a
\emph{diffeomorphism} if it is invertible and its inverse is smooth.\smallskip 

\begin{proposition}
A smooth map between differential spaces is continuous.
\end{proposition}

\begin{proof}
See the proof of proposition 2.1.5 in \cite{sniatycki 2013}
\end{proof}\medskip

A differential space $S$ is \emph{subcartesian} if its topology is Hausdorff and every point 
$x\in S$ has a neighbourhood $U$ diffeomorphic to a subset $V$ of 
$\mathbb{R}^{n}$. It should be noted that $V$ in the definition above may be an arbitrary
subset of $\mathbb{R}^{n}$, and $n$ may depend on $x\in S$. As in the theory
of manifolds, diffeomorphisms of open subsets of $S$ onto subsets of 
$\mathbb{R}^{n}$ are called charts on $S$. The family of all charts is the
complete atlas on $S$. Aronszajn \cite{aronszajn} used the notion of a complete atlas on a
Hausdorff topological space in his definition of subcartesian space.  

\section{Derivations at a point}

Let $S$ be a subcartesian space with differential structure $C^{\infty }(S)$. A \emph{derivation of} $C^{\infty }(S)$ \emph{at a point} $x\in S$ is a linear map 
$v_{x}:C^{\infty }(S)\rightarrow \mathbb{R}$ such that 
\begin{equation}
v_{x}(f_{1}f_{2})=v_{x}(f_{1})f_{2}(x)+f_{1}(x)v_{x}(f_{2})  \label{10}
\end{equation}%
for every $f_{1},f_{2}\in C^{\infty }(S)$. \medskip 

If $M$ is a manifold, then derivations of $C^{\infty }(M)$ satisfy the chain rule. In other words, for every $v\in TM$, $f_{1},...,f_{k}\in C^{\infty }(M)
$ and $F\in C^{\infty }(\mathbb{R}^{k})$%
\begin{equation}
vF(f_{1},...,f_{k})=[\partial _{1}F(f_{1},...,f_{k})(\tau
(v))]vf_{1}+...+[\partial _{k}F(f_{1},...,f_{k})(\tau (v))]vf_{1},
\label{12a}
\end{equation}%
where $\tau : TM \rightarrow M$ is the tangent bundle projection map and 
$\partial _{1},...,\partial _{k}$ are partial derivatives in $\mathbb{R}^{k}$. Our aim in this section is to show that the chain rule is also valid for derivations
of $C^{\infty }(S)$, where $S$ is a subcartesian space.

\begin{lemma}
\label{Lemma 3.2}Let $v$ be a derivation of $C^{\infty }(S)$ at $x\in S$.
For every open neighbourhood $U$ of $x$ and every $f\in C^{\infty }(S),$ $vf$
depends only on the restriction of $f$ to $U$. 
\end{lemma}

\begin{proof}
Let $f_{1},f_{2}\in C^{\infty }(S)$ agree on a neighbourhood $U$ of $x\in S$. 
By lemma 2.2.1 of \cite{sniatycki 2013} there exists a function $h\in
C^{\infty }(S)$ satisfying $h_{\mid V}=1$ for some neighbourhood $V$ of $x$
contained in $U,$ and $f_{\mid W}=0$ for some open set $W$ in $S$ such that $%
U\cup W=S$. Then $h(f_{1}-f_{2})=0$, so that $v[h(f_{1}-f_{2})]=0$. Hence, 
\[
0=v[h(f_{1}-f_{2})]=(vh)(f_{1}-f_{2})(x)+[v(f_{1}-f_{2})]h(x)=vf_{1}-vf_{2}
\]%
because $f_{1}(x)=f_{2}(x)$ and $h(x)=1.$ This implies $vf_{1}=vf_{2}$ as
required. 
\end{proof}\medskip 

Let $\varphi :S\rightarrow R$ be a smooth map between differential spaces
with differential structures $C^{\infty }(S)$ and $C^{\infty }(R)$,
respectively. 

\begin{lemma}
\label{Lemma 3.3}The map $\varphi :S\rightarrow R$ assigns to each
derivation $v$ of $C^{\infty }(S)$ at $x\in S$ a derivation $T\varphi (v)$
of $C^{\infty }(R)$ at $\varphi (x)\in R$ such that, for every $f\in C^{\infty }(R)$, 
\begin{equation}
T\varphi (v) f=v(\varphi ^{\ast }f). \medskip 
\label{13a}
\end{equation}
\end{lemma}

\begin{proof}
For every $f\in C^{\infty }(R)$, $\varphi ^{\ast }f=f\comp \varphi $ is in $%
C^{\infty }(S)$, and we may evaluate the derivation $v$ on $\varphi ^{\ast
}f $. Note that, for $f_{1},f_{2}\in C^{\infty }(R)$ and $c_{1},c_{2}\in 
\mathbb{R},$ 
\begin{align}
T\varphi (v) (c_{1}f_{1}+c_{2}f_{2}) &= v(\varphi ^{\ast
}(c_{1}f_{1}+c_{2}f_{2})=v(c_{1}\varphi ^{\ast }f_{1})+v(c_{2}\varphi ^{\ast}f_{2}) \notag \\
&\hspace{-.5in}=c_{1}v(\varphi ^{\ast }f_{1})+c_{2}v(\varphi ^{\ast }f_{2})=c_{1}\left(
T\varphi (v)\right) f_{1}+c_{2}\left( T\varphi (v)\right) f_{2}. \notag
\end{align}%
Hence, $f \mapsto T\varphi (v)  f$ is a linear mapping of $C^{\infty }(R)$ into itself.
For $f_{1},f_{2}\in C^{\infty }(R)$, equation (\ref{13a}) yields 
\begin{eqnarray*}
T\varphi (v)(f_{1}f_{2}) &=&v((\varphi ^{\ast }f_{1})(\varphi ^{\ast
}f_{2}))=v(\varphi ^{\ast }f_{1})\varphi ^{\ast }f_{2}(x)+\varphi ^{\ast}f_{1}(x)v(\varphi ^{\ast }f_{2}) \\
&=&[T\varphi (v)f_{1}][f_{2}(\varphi (x))]+[f_{1}(\varphi (x))][T\varphi
(v)f_{2}].
\end{eqnarray*}%
Hence $T\varphi (v)$ is a derivation of $C^{\infty }(R)$ at $\varphi (x)$. 
\end{proof}

\begin{theorem}
For each $x$ in a subcartesian space $S$, every derivation $v$ of $C^{\infty
}(S)$ at $x$ satisies the chain rule. In other words, for every $k\in 
\mathbb{N}$, $f_{1},...,f_{k}\in C^{\infty }(S)$ and $F\in C^{\infty }(%
\mathbb{R}^{k}),$%
\begin{equation}
v[F(f_{1},...,f_{k})]=[\partial
_{1}F(f_{1},...,f_{k})(x)]vf_{1}+ \cdots +[\partial
_{k}F(f_{1},...,f_{k})(x)]vf_{k},  \label{14a}
\end{equation}%
where $\partial _{1},...,\partial _{k}$ are partial derivatives in $\mathbb{R%
}^{k}$.
\end{theorem}

\begin{proof}
Since $S$ is subcartesian, there exists a diffeomorphism $\varphi
:W\rightarrow \varphi (W)\subseteq \mathbb{R}^{n},$ where $W$ is an open
neighbourhood of $x$ in $S$. Let $\iota _{W}:W\rightarrow S$ be the
inclusion map. For every $f\in C^{\infty }(S)$, $\iota _{W}^{\ast }f=f|_W \in C^{\infty }(W)$. 
By lemma \ref{Lemma 3.2}, $vf$ is completely
determined by $f|_W$, and equation (\ref{13a}) yields 
\begin{equation}
vf=T\iota _{W}(v)(\iota _{W}^{\ast }f)=T\iota _{W}(v)(f|_W).
\label{15a}
\end{equation}%
Since $\varphi :W\rightarrow \varphi (W)\subseteq \mathbb{R}^{n}$ is a
diffeomorphism, $h=(\varphi ^{-1})^{\ast }f|_W \in 
C^{\infty }(\varphi (W))$ and 
\begin{equation}
(T\iota _{W}(v))f_{\mid W}=[T\varphi (Ti_{W}(v))]h.  \label{16a}
\end{equation}%
Let $\iota _{\varphi (W)}:\varphi (W)\rightarrow \mathbb{R}^{n}$ be the
inclusion map. Since $T\varphi (Ti_{W}(v))$ is a derivation of 
$C^{\infty}(\varphi (W))$ at $\varphi (x)$, $T\iota _{\varphi (W)}(T\varphi (Ti_{W}(v)))$ 
is a derivation of $C^{\infty }(\mathbb{R}^{n})$ at $\varphi (x).$ Without loss of generality, we may assume that the function $h$ in equation (\ref{16a}) is the restriction to $\varphi (W)$ of a function $H\in C^{\infty }(\mathbb{R}^{n})$. Therefore, 
\begin{equation}
\lbrack T\varphi (Ti_{W}(v))]h=[T\iota _{\varphi (W)}(T\varphi
(Ti_{W}(v)))]H.  \label{17a}
\end{equation}

Derivations of $C^{\infty }(\mathbb{R}^{n})$ satisfy the chain rule. If 
$H=F(H_{1},...,H_{k})$, for some $k\in \mathbb{N}$, $H_{1},...,H_{k}\in
C^{\infty }(\mathbb{R}^{n})$ and $F\in C^{\infty }(\mathbb{R}^{k})$, then
equation (\ref{12a}) yields 
\begin{align}
[T\iota _{\varphi (W)}(T\varphi (T{\iota}_{W}(v)))]F(H_{1},...,H_{k}) & = \notag \\
&\hspace{-1.75in} = [\partial _{1}F(H_{1},...,H_{k})(\varphi (x))]
[T\iota _{\varphi (W)}(T\varphi (T{\iota}_{W}(v)))]H_{1}+  \label{ 18a} \\
&\hspace{-1.5in} \cdots +[\partial _{k}F(H_{1},...,H_{k})(\varphi (x))][T{\iota }_{\varphi (W)}
(T\varphi (T{\iota }_{W}(v)))]H_{k}.  \notag 
\end{align}
Therefore, 
\begin{align}
vF(f_{1},...,f_{k}) & =[T{\iota }_{\varphi (W)}(T\varphi
(T{\iota }_{W}(v)))]F(H_{1},...,H_{k})  \notag  \\
&\hspace{-.5in} =[\partial _{1}F(H_{1},...,H_{k})(\varphi (x))][T\iota _{\varphi (W)}
(T\varphi (T{\iota }_{W}(v)))]H_{1}+  \label{19a} \\
&\cdots +[\partial _{k}F(H_{1},...,H_{k})(\varphi (x))][T{\iota }_{\varphi (W)}
(T\varphi (T{\iota}_{W}(v)))]H_{k},  \notag
\end{align}%
where $H_i|_{\varphi (W)}=(\varphi ^{-1})^{\ast }f_i |_W$ for $i=1,\ldots ,k$. 
\end{proof}

\section{The tangent bundle}

Let $T_{x}S$ be the set of all derivations of $C^{\infty }(S)$ at $x\in S$. The set 
$T_xS$ is a real vector space, which is interpreted to be the tangent \emph{tangent space}
to $S$ at $x$. Let $TS $ be the disjoint union of tangent spaces to $S$ at each point $x$ of $S$. In other words,%
\[
TS=\raisebox{-4pt}{\mbox{\LARGE $\amalg $}}_{x\in S}T_{x}S = 
\bigcup_{x \in S}\big( \{ x \} \times T_xS \big) 
\] %
The \emph{tangent bundle projection} is the map 
$\tau :TS\rightarrow S: v = (x,v_x) \rightarrow x$, 
which assigns to each derivation $v_{x}\in TS$ at $x$ the point $x\in S$. The tangent bundle projection enables us to omit the subscript $x$ in the
definition of derivation at a point, and rewrite equation (\ref{10}) in the
form
\begin{equation}
v(f_{1}f_{2})=v(f_{1})f_{2} +f_{1}v(f_2).  \label{11}
\end{equation}

Each function $f\in C^{\infty }(S)$ gives rise to two functions on $TS$, namely, 
\begin{displaymath}
\tau ^{\ast }f :TS\rightarrow \mathbb{R}:v\mapsto f(\tau (v)) 
\end{displaymath}
and 
\begin{displaymath}
\mathrm{d}f :TS\rightarrow \mathbb{R}:v\mapsto \mathrm{d}f(v)=v(f).
\end{displaymath}

The \emph{tangent bundle }of a differential space\ $S$
is $TS$ with differential structure $C^{\infty }(TS)$ generated by the family of functions 
$\{\tau ^{\ast }f, \, \mathrm{d}f \setrule \,  f\in C^{\infty }(S)\}$. This definition of $C^{\infty }(TS)$ ensures that the tangent bundle projection $\tau :TS\rightarrow S$ is smooth. 
The \emph{derived map} of a smooth map $\varphi :S\rightarrow
R$ is $T\varphi :TS\rightarrow TR:v\mapsto T\varphi (v)$, where for
every $f\in C^{\infty }(R)$ $\lbrack T\varphi (v)]f=v(\varphi ^{\ast }f)$, 
see lemma \ref{Lemma 3.3}. If $\tau _{S}:TS\rightarrow S$ and $\tau _{R}:TR\rightarrow R$ are tangent bundle projections, then 
\begin{equation}
\tau _{R}\comp T\varphi =\varphi \comp \tau _{R}.  \label{14}
\end{equation}

\section{Global derivations}

\label{Definition 3.7}A \emph{derivation} of $\mathcal{C}^{\infty }(S)$ is a linear
map $X:\mathcal{C}^{\infty }(S)\rightarrow \mathcal{C}^{\infty }(S):f\mapsto X(f)$ 
satisfying Leibniz's rule%
\begin{equation}
X(f_{1}f_{2})=X(f_{1})f_{2}+f_{1}X(f_{2})  \label{Leibniz}
\end{equation}%
for every $f_{1},f_{2}\in \mathcal{C}^{\infty }(S)$. Let $\mathrm{Der\,}\mathcal{C}^{\infty }(S)$ be the space of derivations of $\mathcal{C}^{\infty }(S)$. It has the structure of a Lie
algebra with the Lie bracket $[X_{1},X_{2}]$ defined by%
\[
\lbrack X_{1},X_{2}](f)=X_{1}(X_{2}(f))-X_{2}(X_{1}(f)) 
\]%
for every $X_{1},X_{2}\in \mathrm{Der\,}\mathcal{C}^{\infty }(S)$ and $f\in 
\mathcal{C}^{\infty }(S).$ Moreover, $\mathrm{Der\,}\mathcal{C}^{\infty }(S)$
is a module over the ring $\mathcal{C}^{\infty }(S)$ and 
\[
\lbrack
f_{1}X_{1},f_{2}X_{2}]=f_{1}f_{2}[X_{1},X_{2}]+f_{1}X_{1}(f_{2})X_{2}-f_{2}X_{2}(f_{1})X_{1} 
\]%
for every $X_{1},X_{2}\in \mathrm{Der\,}\mathcal{C}^{\infty }(S)$ and $%
f_{1},f_{2}\in \mathcal{C}^{\infty }(S)$. If $X$ is a derivation of $\mathcal{C}^{\infty }(S)$,\ then, for every $x\in S$ we have a derivation $X(x)$ of $\mathcal{C}^{\infty }(S)$ at $x$ given
by 
\begin{equation}
X(x):\mathcal{C}^{\infty }(S)\rightarrow \mathbb{R}:f\mapsto X(x)f=(Xf)(x).
\label{X(x)}
\end{equation}%
The derivation $X(x)$ (\ref{X(x)}) is called the \emph{value} of $X$ at $x$. Clearly, the
derivation $X$ is uniquely determined by the collection $\{X(x) \setrule \,  x\in S\}$
of its values at all points of $S$. In order to avoid confusion between
a derivation of $\mathcal{C}^{\infty }(S)$ and a derivation of 
$\mathcal{C}^{\infty }(S)$ at a point in $S$, we shall often refer to the former as a 
\emph{global derivation} of $\mathcal{C}^{\infty }(S)$.

\begin{theorem}
\label{Theorem 3.4a}Let $S$ be a differential subspace of $\mathbb{R}^{n}$
and $X$ a derivation of $\mathcal{C}^{\infty }(S)$. For each $x\in
S\subseteq \mathbb{R}^{n}$, there exists a neighbourhood $U$ of $x$ in $%
\mathbb{R}^{n}$ and a vector field $Y$ on $\mathbb{R}^{n}$ such that%
\[
X(F|_S)|_{U\cap S}=(Y(F))|_{U\cap S} 
\]%
for every $F\in \mathcal{C}^{\infty }(\mathbb{R}^{n}).$
\end{theorem}

\begin{proof}
Let $Z$ be a derivation of $\mathcal{C}^{\infty }(S)$ at $x\in S\subseteq 
\mathbb{R}^{n}.$ For each $F\in \mathcal{C}^{\infty }(\mathbb{R}^{n})$ the
restriction $F|_S$ of $F$ to $S$ is in $\mathcal{C}^{\infty }(S)$. It
is easy to see that the map 
$\mathcal{C}^{\infty }(\mathbb{R}^{n})\rightarrow \mathbb{R}:F\mapsto Z(F|_S)$ is a derivation at $x$ of $\mathcal{C}^{\infty }(\mathbb{R}^{n})$.\smallskip

We denote the natural coordinate functions on $\mathbb{R}^{n}$ by 
$x^{1},...,x^{n}:\mathbb{R}^{n}\rightarrow \mathbb{R}$. Every derivation $Y$
of $\mathcal{C}^{\infty }(\mathbb{R}^{n})$ is of the form 
$\sum_{i=1}^{n}F^{i}\frac{\partial }{\partial x^{i}},$ where $F^{i}=Y(x^{i})$
for $i=1,...n.$ Let $X$ be a derivation of $\mathcal{C}^{\infty }(S)$ and 
$F\in \mathcal{C}^{\infty }(\mathbb{R}^{n})$. For each $x\in S$, the
derivation $X(x)$ of $\mathcal{C}^{\infty }(S)$ at $x$ gives a derivation of 
$\mathcal{C}^{\infty }(\mathbb{R}^{n})$ at $x$. Hence, 
\begin{align}
X(F|_S )(x) &= X(x)(F|_S )=
\sum_{i=1}^{n}\frac{\partial F}{\partial x^{i}}(x)(X(x)(x^i|_S )) \notag \\
&=\sum_{i=1}^{n}\frac{\partial F}{\partial x^{i}}(x)(X(x^i|_S))(x) = 
\big( \sum_{i=1}^{n}X(x^i|_S) \frac{\partial F}{\partial x^{i}}\big) (x) \notag
\end{align}%
for every $x\in S$. For $i=1,...,n,$ the coefficients $X(x^i|_S)$ are in 
$\mathcal{C}^{\infty }(S)$. Since $S$ is a differential subspace of $\mathbb{R}^{n}$,
for each $x\in S$ there exists a neighbourhood $U$ of $x$ in $\mathbb{R}^{n}$
and functions $F^{1},...,F^{n}\in \mathcal{C}^{\infty }(\mathbb{R}^{n})$
such that $X(x^i|_S)|_{U\cap S}=F^i |_{U\cap S}$ for each $i=1,\ldots ,n$. Hence, 
\[
X(F|_S)|_{U\cap S}=
\left( \sum_{i=1}^{n}F^{i}\frac{\partial F}{\partial x^{i}}\right) \rule[-10pt]{.5pt}{24pt}
\raisebox{-9pt}{$\, \scriptscriptstyle U\cap S$} .
\]%
Since $F^{1},...,F^{n}$ are smooth functions on $\mathbb{R}^{n}$, it follows
that $Y=\sum_{i=1}^{n}F^{i}\frac{\partial }{\partial x^{i}}$ is a vector
field on $\mathbb{R}^{n}.$\smallskip 
\end{proof}

We can rephrase theorem \ref{Theorem 3.4a} by saying that every derivation
on a differential subspace $S$ of $\mathbb{R}^{n}$ can be locally extended
to a vector field on $\mathbb{R}^{n}$. Suppose that $S$ is closed. In this
case, we can use a partition of unity on $\mathbb{R}^{n}$ to extend every
derivation of $\mathcal{C}^{\infty }(S)$ to a global vector field on $%
\mathbb{R}^{n}$. \medskip 

A \emph{section} of the tangent bundle projection $\tau :TS\rightarrow S$ is a
smooth map $\xi :S\rightarrow TS$ such that $\tau \comp \xi =\mathrm{id}_{S}$. 
Let $\mathcal{S}^{\infty }(TS)$ be the space of
sections of the tangent bundle projection $\tau :TS\rightarrow S$. Since the differential structure $C^{\infty }(TS)$ is generated by the collection of functions 
$\{\tau ^{\ast }f,\dee f \setrule \,  f\in C^{\infty}(S)\}$, it follows that a section $\xi :S\rightarrow TS$ has to satisfy the conditions that $\xi ^{\ast }(\tau ^{\ast }f)$ and 
$\xi ^{\ast }(\mathrm{d}f) $ are in $C^{\infty }(S)$ for every $f\in C^{\infty }(S).$ The first
condition holds automatically because 
\begin{equation}
\xi ^{\ast }(\tau ^{\ast }f)=(\tau ^{\ast }f)\comp \xi =f\comp \tau \comp \xi 
=f \comp \mathrm{id}_{S}=f.  \label{0.4}
\end{equation}%
On the other hand, for $x\in S$,  
\begin{equation}
(\xi ^{\ast }(\dee f))(x)=((\dee f)\comp \xi )(x)= 
\langle \dee f \mid \xi (x) \rangle =\xi (x)f.  \label{0.5}
\end{equation}%

\begin{proposition}
\label{Proposition 3.7}Every global derivation $X$ of $C^{\infty }(S)$
defines a section 
\begin{equation}
X:S\rightarrow TS:x\mapsto X(x),   \label{0.6}
\end{equation}%
where $X(x)f=(Xf)(x)$ for every $f\in C^{\infty }(S)$ and every $x \in S$. 
\end{proposition}

\begin{proof}
The section $X:S\rightarrow TS$, defined by equation (\ref{0.6}), satisfies
equation (\ref{0.5}) because $X^{\ast }(\mathrm{d}f)=X(f)\in \mathrm{Der}\, C^{\infty }(S)$ by definition of a global derivation. Conversely, if $\xi :S\rightarrow TS$ is a section, then equation (\ref{0.5}) implies that $\xi (f)=\xi ^{\ast }(\mathrm{d}f)\in \mathrm{Der}\, C^{\infty }(S)$ for every 
$f\in C^{\infty }(S)$. Hence, $\xi :f\mapsto \xi (f)$ is a global derivation of $C^{\infty }(S)$.
\end{proof}\smallskip

Equation (\ref{0.6}) gives a bijection between the space $\mathcal{S}^{\infty }(TS)$ of sections of the tangent bundle projection and the space $\mathrm{Der}\, C^{\infty }(S)$. hence
proposition \ref{Proposition 3.7} leads to identification of global
derivations of $\mathcal{C}^{\infty }(S)$ with the correspnding sections of
the tangent bundle.$\smallskip $

Let $c:I\rightarrow S$ be a smooth map of an interval $I$ in $\mathbb{R}$ containing $0$ to
a differential space $S$. We say that $c$ is an \emph{integral curve} of a
derivation $X$ of $\mathcal{C}^{\infty }(\mathbb{R})$ \emph{starting at} $x_{0}\in S$ if $x_{0}=c(0)$ and 
\begin{equation}
\frac{\dee }{\dee t}f(c(t))=X(f)(c(t))  \label{curve}
\end{equation}%
for every $f\in \mathcal{C}^{\infty }(S)$ and every $t\in I$. In other
words, $c:I\rightarrow S$ is an integral curve of $X$ if $Tc(t)=X\comp c(t)$
for every $t\in I$. In the case of subcartesian spaces, we have to allow the interval $I$ to be a single point in $\mathbb{R}$. In other words, the
tangent vector $X(x_{0})\in T_{x_{0}}S$ may be considered as an integral
curve $c:\{ 0 \}\rightarrow S: 0 \mapsto x_{0}$ of $X$ starting at $x_{0}$. \medskip

Integral curves of a given derivation $X$ of $\mathcal{C}^{\infty }(S)$ can
be ordered by inclusion of their domains. In other words, if $%
c_{1}:I_{1}\rightarrow S$ and $c_{2}:I_{2}\rightarrow S$ are two integral
curves of $X$ and $I_{1}\subseteq I_{2}$, then $c_{1}\preceq c_{2}$. An 
integral curve $c_{1}:I\rightarrow S$ of $X$ is \emph{maximal} if $c_{1}\preceq
c_{2}$ implies that $c_{1}=c_{2}$.\medskip 

\noindent \textbf{Example 7} Let $\mathbb{Q}$ be the set of rational numbers in 
$\mathbb{R}$. Then $\mathcal{C}^{\infty }(\mathbb{Q})$ consists of restrictions to $\mathbb{Q}$
of smooth functions on $\mathbb{R}$. Since $\mathbb{Q}$ is dense in $\mathbb{R}$, it follows that every function $f\in \mathcal{C}^{\infty }(\mathbb{Q})$
extends to a unique smooth function on $\mathbb{R}$ and every derivation of 
$\mathcal{C}^{\infty }(\mathbb{R})$ induces a derivation of 
$\mathcal{C}^{\infty }(\mathbb{Q})$. Let $X$ be the derivation of 
$\mathcal{C}^{\infty }(\mathbb{Q})$ induced by the derivative $\frac{\dee }{\dee x}$ on 
$\mathcal{C}^{\infty }(\mathbb{R})$. In other words, for every $f\in \mathcal{C}^{\infty
}(\mathbb{Q})$ and every $x_{0}\in \mathbb{Q},$ 
\[
(Xf)(x_{0})=\lim_{x\rightarrow x_{0}}\frac{f(x)-f(x_{0})}{x - x_{0}}, 
\]%
where the limit is taken over $x\in \mathbb{Q}$. Since no two distinct
points in $Q$ can be connected by a continuous curve, it follows that for each 
$x\in \mathbb{Q}$, the tangent vector $X(x)\in T_{x}\mathbb{Q}$ is the
maximal integral curve of $X$ through $x$.\quad \mbox{\tiny $\blacksquare $} \medskip 

\noindent \textbf{Definition} Let $\tau :TS\rightarrow S$ be the tangent bundle projection. Let $X$ be a derivation of the differential structure $\mathcal{C}^{\infty }(S)$ of a subcartesian space $S$. Let $x_{0}$ be a point in $S$, and $I$ be an interval in $\mathbb{R}$ containing $0\in 
\mathbb{R}$ or $I=\{0\}$. A \emph{lifted integral curve} of $X$ \emph{starting at} $x_{0}$ is a map 
$\gamma :I\rightarrow TS$ such that $\gamma (0)=X(x_{0})$ and 
\begin{equation}
\frac{\dee }{\dee t}f \big( \tau (\gamma (t))\big)=X(f)\big( \tau (\gamma (t)) \big)
\label{integral curve} 
\end{equation}
for every $f\in C^{\infty }(S)$ and $t\in I$, if $I\neq \{0\}$. \medskip 

\noindent If $I\neq \{0\}$, then setting $c=\tau \comp \gamma $ we recover the definition for an integral curve of a derivation given in equation (\ref{curve}). If $I = \{ 0 \}$, then $\gamma $ is 
a lifted integral curve of $X$ starting at $x_0$, because $\gamma (0) = X(c(0)) = X(x_0)$. 
Our extension of this definition to subcartesian spaces
requires lifting the curve $c:I\rightarrow S$ to the tangent bundle leading
to $\gamma :I\rightarrow TS$, in order to make sense of the condition 
$\gamma (0)=X(c(0))$. 
\addtocounter{theorem}{1}

\begin{theorem} Let $S$ be a subcartesian space and let $X$ be a derivation
of $\mathcal{C}^{\infty }(S)$. For every $x\in S$, there exists a unique maximal integral curve 
${\gamma }_x$ of $X$ starting at $x$. 
\end{theorem}

\begin{proof}
\par \noindent i) \textbf{Local existence.} For $x\in S$, let $\varphi $ be a
diffeomorphism of a neighbourhood $V$ of $x$ in $S$ onto a differential
subspace $R$ of $\mathbb{R}^{n}$. Let $Z=\varphi _{\ast }X|_V$ be a
derivation of $\mathcal{C}^{\infty }(R)$ obtained by pushing forward the
restriction of $X$ to $V$ by $\varphi $. In other words, 
\begin{equation}
Z(f)\comp \varphi =X|_V (f\comp \varphi ) 
\label{19}
\end{equation}%
for all $f\in \mathcal{C}^{\infty }(R)$. Without loss of generality, we may
assume that there is an extension of $Z$ to a vector field $Y$ on 
${\R }^{n}$.\smallskip 

Let $z=\varphi (x)$ and let $c_{0}$ be a standard integral curve in $\mathbb{R}^{n}$ of
the vector field $Y$ such that $c_{0}(0)=z$. Let $I_{x}$ be the connected
component of $c_{0}^{-1}(R)$ containing $0$ and let $c:I_{x}\rightarrow R$ the
curve in $R$ obtained by the restriction of $c_{0}$ to $I_{x}$. Clearly, $c(0)=z$. 
If $I_{x}\neq \{0\}$, then for each $t_{0}\in I_{x}$ and each $f\in 
\mathcal{C}^{\infty }(R)$ there exists a neighbourhood $U$ of $c(t_{0})$ in 
$R$ and a function $F\in \mathcal{C}^{\infty }(\mathbb{R}^{n})$ such that 
$f|_U =F|_U$ and 
\begin{align}
\frac{\dee }{\dee t}\rule[-8pt]{.5pt}{20pt}\raisebox{-7pt}{$\, \scriptscriptstyle t = t_0$} 
\hspace{-7pt}f(c(t)) &= 
\frac{\dee}{\dee t}\rule[-8pt]{.5pt}{20pt}\raisebox{-7pt}{$\, \scriptscriptstyle t = t_0$} \hspace{-7pt} F(c(t)) =Y(F)(c(t_{0})) \notag \\
&= Y(F)|_U( c(t_{0}) )=Z(f)(c(t_{0})). \notag 
\end{align}%
Since $I_{x}\neq \{0\}$ is a connected subset of $\mathbb{R}$ containing $0$, it is an interval. 
So $c_{x}=\varphi ^{-1}\comp c:I_{x}\rightarrow
V\subseteq S$ satisfies $c_{x}(0)=\varphi ^{-1}(c(0))=\varphi ^{-1}(z)=x$.
Moreover, for every $t\in I_{x}$ and $h\in \mathcal{C}^{\infty }(S)$, we get 
$f=h\comp \varphi ^{-1}\in \mathcal{C}^{\infty }(R)$ and 
\begin{align}
\frac{\dee }{\dee t}h\big( c_{x}(t) \big) &=\frac{\dee }{\dee t}h\big(\varphi ^{-1}\big( c(t)) \big)= 
\frac{\dee}{\dee t}(h\comp \varphi ^{-1})\big( c(t) ) \notag \\
&=\frac{\dee}{\dee t}f(c(t)) = Z(f)(c(t)) \notag \\
&= Z(h\comp \varphi ^{-1})(\varphi \comp c_{x}(t))=X(h)\big( c_{x}(t) \big) \notag .
\end{align}%
This implies that the map ${\gamma }_x:I_{x}\rightarrow TS:t\mapsto X(c_{x}(t))$ 
is a lifted integral curve of $X$ starting at $x$ if $I_{x}\neq \{0\}$. It is also
an integral curve of $X$ starting at $x$ when $I_{x}=\{0\}$, because 
${\gamma}_x (0)=X(x)$. \smallskip 

\noindent ii) \textbf{Smoothness.} From the theory of differential
equations it follows that the integral curve $c_{0}$ in ${\R}^{n}$ of a smooth
vector field $Y$ is smooth. Hence, $c=c_0|_{I_x}$ is smooth. Since 
$\varphi $ is a diffeomorphism of a neighbourhood of $x$ in $S$ to $R$, its
inverse $\varphi ^{-1}$ is smooth and the composition 
$c_{x}= {\varphi }^{-1}\comp c$ is smooth. Since $X$ is a derivation, it gives rise to a smooth
section $X:S\rightarrow TS$ of the tangent bundle projection 
$\tau :TS\rightarrow S$. Moreover the composition 
$\gamma _{x}=X\comp c_{x}$ is smooth.\smallskip 

\noindent iii) \textbf{Local uniqueness.} This follows from the
local uniqueness of solutions of first order differential equations in $\mathbb{R}^{n}$. \smallskip 

\noindent iv) \textbf{Maximality.} Suppose that $p \le 0 \le q$ are the ends of the domain 
$I_x$ of the integral curve $c_x$ of $X$ starting at $x$ obtained in section i). If 
$p =q =0$ and ${\gamma }_x = X \comp c_x$ cannot be extended to a larger interval, then  
${\gamma }_x$ is maximal. If $q > 0$ and $q = \infty$ or $\lim_{t\nearrow q}c_x(t)$ does 
not exist, then the curve $c_x$ does not extend beyond $q$. If $x_1 = 
\lim_{t \nearrow q}c_x(t)$ exists, then $x_1$ is unique since the topology of $S$ is Hausdorff. 
We can repeat the construction of section i) by starting at the point $x_1$. In this way we 
obtain an integral curve $c^1_{x_1}: I_1 \rightarrow S$ of $X$ starting at $x_1$. Let 
${\widetilde{I}}_1 = I \cup \{ t = q +s \in \R \setrule \, s \in I_1 \cap [0, \infty) \} $ and let 
${\widetilde{c}}_1: {\widetilde{I}}_1 \rightarrow S$ be given by ${\widetilde{c}}_1(t) = 
c_x(t)$ if $t \in I$ and ${\widetilde{c}}_1(t) = c^1_{x_1}(t-q)$ if $t \in \{ q+s \in \R \setrule \, 
s \in I_1 \cap [0, \infty) \} $. Clearly the curve ${\widetilde{c}}_1$ is continuous. Since 
$x_1 = \lim_{t \nearrow q}c_x(t)$, it follows that the left end $p_1$ of $I_1$ is 
strictly less than zero. Hence, the restriction of the curve $c_x$ to the interval 
$(\max(p,p_1) +q, q)$ differs from the restriction of $c^1_{x_1}$ to the interval 
$(\max(p,p_1), 0)$ by the reparametrization $t \mapsto t-q$. Since the curves 
$c_x$ and $c^1_{x_1}$ are smooth, it follows that the curve ${\widetilde{c}}_1$ is smooth. Let 
$q_1$ be the right end of the interval $I_1$. If $q_1 \in I$ and either 
$q_1= \infty$ or $\lim_{t\nearrow q_1}c^1_x(t) $ does not exist, then the curve curve 
${\widetilde{c}}_1$ does not extend beyond $q_1$. Otherwise, we can extend 
${\widetilde{c}}_1$ by an integral curve $c_2$ of $X$ through $x_2 = 
\lim_{t \nearrow q_1}{\widetilde{c}}_1(t)$. 
Continuing this process, we obtain a maximal extension for $t \ge 0$. In a similar 
way we can construct a maximal extension for $t \le 0$. \smallskip 

\noindent v) \textbf{Global uniqueness.}  Let $c: I \rightarrow S$ and $c': I' \rightarrow S$ 
be maximal integral curves of $X$ starting at $x$. Let 
$T^{+} = \{ t \in I \cap I' \setrule \, t > 0 \, \, \mathrm{and} \, \, c(t) \ne c'(t) \} $. Suppose that 
$T^{+} \ne \varnothing $. Since $T^{+}$ is bounded from below by $0$, there is a greatest 
lower bound $\ell $ of $T^{+}$. This implies that $c(t) = c'(t)$ for every $0 \le t \le \ell$ and  
for every $\varepsilon > 0$ there is a $t_{\varepsilon }$ with  
$\ell < t_{\varepsilon } < \ell + \varepsilon$ such that $c(t_{\varepsilon }) \ne c'(t)$. Let 
$x_{\ell} = c(\ell ) = c'(\ell )$ and let $c_{\ell }:I_{\ell } \rightarrow S $ be an integral curve 
of $X$ starting at $x_{\ell }$ as constructed in i). Let $q_{\ell }$ be the right end of $I_{\ell }$. 
If $q_{\ell }>0$, the local uniqueness of integral curves implies that 
$c(t) = c'(t) =c_{\ell }(t - \ell )$ for all $\ell \le t \le \ell +q_{\ell }$. This contradicts the fact 
that $\ell $ is the greatest lower bound of $T^{+}$. If $q_{\ell }=0$, then there is no 
extension of $c_{\ell }$. Let $q$ and $q'$ be the right end of $I$ and $I'$, 
respectively. Since $c$ and $c'$ are maximal integral curves of $X$, it follows that 
$q = q' = \ell $. Hence the set $T^{+}$ is empty, which is a contradiction. A similar 
argument shows that $T^{-} = \{ t \in I \cap I' \setrule \, t <0 \, \, \mathrm{and} \, \, 
c(t) \ne c'(t) \} = \varnothing $. Therefore $c(t) = c'(t)$ for all $t \in I \cap I'$. If 
$I \ne I'$, then this contradicts the fact that $c$ and $c'$ are maximal. Hence 
$I = I'$ and $c = c'$. \end{proof}  

\vspace{-.1in}
\section{Vector fields}

Vector fields on a manifold $M$ are not only derivations of $\mathcal{C}^{\infty }(M)$ 
but they also generate local one-parameter groups of local
diffeomorphisms of $M.$ On a subcartesian space $S$, not all derivations of $%
\mathcal{C}^{\infty }(S)$ generate local one-parameter groups of local
diffeomorphisms of $S,$ see example 6. We reserve the term vector field for
derivations of $\mathcal{C}^{\infty }(S)$ that generate local one-parameter
groups of local diffeomorphisms of $S$. More formally, we adopt the
following definition. A \emph{vector field} on a subcartesian space $S$ is a derivation $X$ of 
$\mathcal{C}^{\infty }(S)$ such that for every $x_{0}\in S,$ there exists a
neighbourhood $U_{x_{0}}$ of $x_{0}\in S$ and $\varepsilon _{x_{0}}>0$ such
that, for every $x\in U_{x_{0}}$, the interval 
$(-\varepsilon _{x_{0}},\varepsilon _{x_{0}})$ is contained in the domain $I_{x}$ of the lifted integral curve $\gamma _{x}:I_{x}\rightarrow TS$ of $X$ and the map 
\begin{equation}
\mathrm{e}^{tX}:U_{x_{0}}\rightarrow S:x\mapsto \mathrm{e}^{tX}(x)=\tau
\comp \gamma _{x}(t)  \label{20}
\end{equation}%
is defined for every $t\in (-\varepsilon _{x_{0}},\varepsilon _{x_{0}})$ and
is a diffeomorphism of $U_{x_{0}}$ onto an open subset $\mathrm{e}%
^{tX}(U_{x_{0}})$ of $S$. \smallskip 

Note that if $X$ is a vector field on $S$, for every $x\in S$, the map $%
c_{x}:I_{x}\rightarrow S:t\rightarrow \mathrm{e}^{tX}(x)$ is an integral
curve of $S$ satisfying equation (\ref{curve}).\footnote{$\mathrm{e}^{tX}(x)$
is a compact version of the notation $(\exp tX)(x)$ used in \cite{sniatycki
2013}} Therefore, if we were only interested in vector fields on $S$, we
could use the definition of integral curves given by equation (\ref{curve}).
We have introduced the notion of lifted integral curves to obtain theorem 8, which 
ensures the existence and uniquenness of maximal lifted integral curves of
derivations of $C^{\infty}(S)$. Theorem 8 is replaces theorem 3.2.1 in \cite{sniatycki 2013},
which is incorrect. Our discussion shows that proofs of all results in 
\cite{sniatycki 2013} regarding vector fields on subcartesian spaces are not
affected by errors in theorem 3.2.1. In particular, all results in section
3.4 and chapter 4 of \cite{sniatycki 2013} are valid. \bigskip

\noindent {\large \textbf{Acknowledgement}} \medskip
\par \noindent The authors would like to thank Larry Bates and Eugene Lerman for stimulating discussions.

\end{document}